\documentclass[a4paper,8pt]{article}

\usepackage{a4wide}
\usepackage{authblk}
\usepackage{cite}
\usepackage[british]{babel}
\usepackage[utf8]{inputenc}
\usepackage{xcolor}
\usepackage{graphicx,subfigure}
\usepackage{epstopdf}
    \graphicspath{{./}{figures/}}
\usepackage[colorlinks=true,linkcolor=blue,citecolor=blue]{hyperref}
\usepackage{xcolor}
\usepackage{amsfonts,amsmath,amsthm}

\newtheorem{theorem}{Theorem}[section]

\newtheorem{lemma}[theorem]{Lemma}

\theoremstyle{remark}\newtheorem{remark}[theorem]{Remark}
\newenvironment{equations}{\equation\aligned}{\endaligned\endequation}

\newcommand{\be}{\begin{equation}}
\newcommand{\ee}{\end{equation}}

\newcommand{\R}{\mathbb{R}}

\newcommand{\fer}[1]{(\ref{#1})}
\allowdisplaybreaks

\begin{document}
\title{{Supercritical Fokker-Planck equations for consensus dynamics: large-time behaviour and weighted Nash-type  inequalities}}

\author[1,2]{Giuseppe Toscani \thanks{\texttt{giuseppe.toscani@unipv.it}}}
\author[1]{Mattia Zanella \thanks{\texttt{mattia.zanella@unipv.it}}}

\affil[1]{Department of Mathematics ``F. Casorati'', University of Pavia, Italy} 
\affil[2]{Institute of Applied Mathematics and Information Technologies, National Research Council (CNR), Pavia, Italy}

\date{\today}

\maketitle
\abstract{We study the main properties of the solution of a Fokker-Planck equation characterized by a variable diffusion coefficient and a polynomial superlinear drift, modeling the formation of consensus in a large interacting system of individuals.  The Fokker-Planck equation is derived from the kinetic description of the dynamics of a quantum particle system, and in presence of a high nonlinearity in the drift operator, mimicking the effects of the mass in the alignment forces,  allows for steady states similar to a Bose-Einstein condensate. {The main feature of this Fokker–Planck equation is the presence of a variable diffusion coefficient, a nonlinear drift and boundaries, which introduce new challenging mathematical problems in the study of its long-time behavior. In particular, propagation of regularity is shown as a consequence of new weighted Nash and Gagliardo--Nirenberg inequalities.}


\section{Introduction}

In the last decades, several models have been proposed to describe the emergence of consensus in large interacting particle systems \cite{DG,HK}. Typically, artificial  consensus-type phenomena are determined by the balance between random dynamics and compromise forces, depending on the distance between the states of two interacting particles. When the number of particles become infinite, it is possible to derive kinetic and mean-field models that allow to study the evolution of the density function expressing the fraction of particles/agents  around an emerging state \cite{APTZ,BHW,D_etal,MT,PTTZ,T}. Applications of such concepts found relevant applications in a heterogeneity of fields, see e.g. \cite{CFRT,CS,DFWZ}. Among them, kinetic equations have been derived for large particle systems which, in the zero-diffusion limit, provide exponential in time collapse of the support of the density function, i.e. blow-up of the solution.   

In this work, we concentrate on the evolution of the density function $f(w,t)$ solution of the one-dimensional nonlinear Fokker-Planck-type equations for consensus phenomena studied in \cite{CDTZ}
\begin{equation}
\label{eq:FP}
\dfrac{\partial}{\partial t} f(w,t) = \dfrac{\partial}{\partial w} \left[ wf(w,t)(1+\beta H^\alpha(w) f^\alpha(w,t)) + \lambda \dfrac{\partial}{\partial w} (H(w)f(w,t))\right], \qquad w \in [-1,1]
\end{equation}
with $\alpha,\beta\ge0$, $H(w) = 1-w^2$ and $f(w,0) = f_0(w) \in L^1(-1,1) \cap L^\infty(-1,1)$. The above equation is complemented by no-flux boundary conditions. In suitable regimes, the partial differential equation \eqref{eq:FP} can be formally derived from a nonlinear Boltzmann-type model, in which the interaction frequency is weighted by the kinetic density, see \cite{CDTZ}. In detail, eq. \eqref{eq:FP} is intended to describe the impact of population's size on the evolution of a system of interacting particles. The large time distribution of \eqref{eq:FP}  reads
\begin{equation} \label{steady}
f_\infty(w) = C\dfrac{(1+w)^{\frac{1}{2\lambda}}(1-w)^{\frac{1}{2\lambda}}}{(1-\beta C^\alpha(1+w)^{\frac{\alpha}{2\lambda}}(1-w)^{\frac{\alpha}{2\lambda}} )^{\frac{1}{\alpha}}},\qquad C = \beta^{-1/\alpha}.
\end{equation}
The steady state in \fer{steady} is such that as $\alpha>2$ a critical mass $\mu_c>0$ exists, see \cite{CDTZ}, and when the critical mass is exceeded, condensation of the exceeding mass around the point $w=0$ appears. 
{We recall that similar phenomena have been studied for the dynamics of quantum indistinguishable particles relaxing towards the classical Bose--Einstein distribution through the so-called Kaniadadakis-Quarati equation \cite{KQ}, see also related works \cite{BAGT,CRS,T2}. In this direction we mention the recent results in  \cite{Hopf} related to a Fokker--Planck equation generalizing the drift term in Kaniadakis--Quarati equation, and the works on one-dimensional advection-diffusion equations with superlinear drift \cite{BL,guidolin}, where it is considered the case with constant diffusion.}

When $\alpha =0$, the Fokker--Planck equation \fer{eq:FP} reduces to the classical equation for opinion formation introduced in \cite{T}, characterized by a linear drift. In this case, the linearity of the equation allows for a variety of powerful mathematical methods which guarantee both the regularity of the solution, and its exponential convergence towards the equilibrium density in relative entropy \cite{FPTT2}.  Anyway, when $\alpha >0$, these methods are no more applicable, and new technical results need to be considered. Furthermore, the methods previously applied to Fokker-Planck type equations linked to a steady state in the form of the Bose-Einstein distribution, which are characterized or by a linear diffusion \cite{CRS,BAGT,T2}, or by a linear drift \cite{FLST}, have to be reconsidered due to the presence of the variable diffusion coefficient $H(w)$. We highlight that the regularity results related to Fokker-Planck type equations with linear diffusion and a general nonlinear drift recently considered in \cite{BL,guidolin,Hopf} can not be immediately extended to the present situation, due to the presence of the variable coefficient of diffusion. We further remark  that the regularity of the solution of equation \fer{eq:FP} is closely related to the value of the parameter $\alpha\ge0$, as  it is easily understandable by looking at its equilibrium density, which, as shown in \cite{CDTZ} presents a condensate when the total density of individuals exceeds a fixed critical level and $\alpha > 2$.

Looking at the existing literature, the analysis of the linear case in \cite{FPTT2} shows that in the case of the variable diffusion coefficient $H(w)$ one can still prove a logarithmic Sobolev inequality, with the correct weight, in the form
\[
\int_{-1}^1 f^2(w)\log f^2(w)\, dw + \log 2 \le 2\int_{-1}^1 H(w)[f'(w)]^2\, dw.
\]
This result, coupled to the general Poincar\'e inequality with weight proven in \cite{FPTT} and subsequently generalized in \cite{FPTT1,ToR}, given in our case by
\[
\int_{-1}^1 \left[f(w)-\frac 12 \int_{-1}^1f(v)\,dv) \right]^2 \, dw  \le \frac 12\int_{-1}^1 H(w)[f'(w)]^2\, dw,
\]
suggests that, in the presence of a uniform density on the interval $[-1,1]$, one can  obtain a weighted version of Nash and Gagliardo-Nirenberg inequalities, which allows, like in the case of a constant coefficient of diffusion, to a precise study of the regularity of the solution to \fer{eq:FP}. \\
 In what follows, we show that this is indeed the case, so that the regularity of the solution to \eqref{eq:FP} for any $\alpha \in [0,2]$ follows by a suitable extension of both Nash and Gagliardo-Nirenberg inequalities in presence of weight. These new results allow a detailed study of the regularity of the solutions as well as the conditions for its blow-up in finite time. { Even if these new inequalities constitute the main novelty of the paper, to maintain an easy-to-read presentation, we will postpone a detailed proof, which include a number of  technical details,  into an  Appendix.} \\
 In presence of a superlinear drift characterized by $\alpha>2$, where the regularity of the problem is  lost in a classical $L^2$ space,  we introduce a class of weighted spaces $L^q_p((-1,1))$ where regularity is proven for a sufficiently large diffusion coefficient for $q=2$. Hence, since the solution remains bounded in the new weighted space,  formation of Dirac masses is not in general allowed  for any finite fixed $\alpha>2$. 

In more detail, the manuscript is organized as follows: in Section \ref{sec:regularity} we study the regularity of the solution for $\alpha\le 2$.  Further, the loss of $L^2$ regularity when $\alpha >2$ is proven in Section \ref{sect:2} by showing that a finite-time blow-up of the solution is possible  in presence of a sufficiently small initial energy or a sufficiently high initial mass. Last,  Section \ref{sect:4} is concerned with a class of spaces where the regularity is preserved also for $\alpha>2$. The technical proof of the new Nash and Gagliardo-Nirenberg inequalities with weight  are detailed  in the Appendix \ref{appendix}. 

\section{Regularity of the solution when $\alpha \le2$}\label{sec:regularity}

Following the analysis in \cite{LBL} we  observe that the Fokker-Planck-type equation \eqref{eq:FP} may be rewritten in various equivalent formulations. Indeed, since for all $t\ge0$ and $w \in (-1,1)$ 
\[
\dfrac{\partial^2}{\partial w^2} \left( H(w)f(w,t)\right) = -2 \dfrac{\partial}{\partial w}\left(wf(w,t) \right) + \dfrac{\partial}{\partial w}(H(w)\dfrac{\partial}{\partial w}f(w,t)),
\]
we have that \eqref{eq:FP} admits the following equivalent formulation
\begin{equation}
\label{eq:FP1}
\dfrac{\partial}{\partial t} f(w,t) =\dfrac{\partial}{\partial w} \left\{ wf(w,t)\left[(1-2\lambda)+\beta H^\alpha(w) f^\alpha(w,t)\right] + \lambda H(w) \dfrac{\partial}{\partial w} f(w,t)\right\}. 
\end{equation}
Similarly, from 
\[
\dfrac{\partial^2}{\partial w^2}(H(w) f) = - 2f -4w\dfrac{\partial}{\partial w} f + H(w)\dfrac{\partial^2}{\partial w^2} f
\]
we have that \eqref{eq:FP} can be rewritten as
\begin{equation}
\label{eq:FP2}
\dfrac{\partial}{\partial t}f(w,t) +2\lambda f + 4\lambda w \dfrac{\partial}{\partial w}f = \dfrac{\partial}{\partial w} \left[ wf(w,t)(1+\beta H^\alpha(w) f^\alpha(w,t))\right]+ \lambda H(w)\dfrac{\partial^2}{\partial w^2} f(w,t). 
\end{equation}
In the rest of this section we will adopt the second formulation of the Fokker-Planck model as detailed in equation \eqref{eq:FP1}. In the rest of this Section, we will assume that equation \fer{eq:FP1} has a smooth solution. We do not wish to be too precise here about the meaning of a smooth solution; this just means that all the integrability and differentiability properties which are needed in the proof of the forthcoming results are satisfied. Moreover, we will assume that the positivity of the solution is propagated in time.  

To reduce the number of parameters, we can resort to the case with a diffusion parameter equal to one. Indeed, in the time scale $\tau = \lambda t$, the model \eqref{eq:FP} can be rewritten as 
\begin{equation}
\label{eq:FP1_scaled}
\dfrac{\partial}{\partial \tau} f(w,\tau) = \dfrac{\partial}{\partial w} \left\{ wf(w,\tau)\left[\lambda^*+\beta^* H^\alpha(w) f^\alpha(w,\tau)\right] +  H(w) \dfrac{\partial}{\partial w} f(w,\tau)\right\}
\end{equation}
being $\lambda^* = \frac{1-2\lambda}{\lambda} \in \mathbb R$ and $\beta^* = \frac{\beta}{\lambda}$. 

The smoothness assumption, coupled with the no-flux boundary conditions, guarantees that, if the initial density $f_0(w)\ge 0$ belongs to $L^1(-1,1)$, the $L^1$-norm of the nonnegative solution is preserved in time (conservation of mass). Therefore, for each $\tau >0$ we have
\[
\int_{-1}^1 f(w,\tau)\, dw = \int_{-1}^1 f_0(w)\, dw.
\]
Let $0<\alpha <2$, and let us suppose now that  $f_0(w)\ge 0$ belongs to $L^1(-1,1)\cap L^2(-1,1)$, and let us recover the evolution in time of the $L^2$-norm of the solution. Integration by parts gives
\begin{equations}\label{L2}
&\frac d{d\tau}\int_{-1}^1 f^2(w,\tau)\, dw = -2 \int_{-1}^1 H(w)\left| \frac{\partial}{\partial w} f(w,\tau)\right|^2\, dw + \\
&+\lambda^*\int_{-1}^1 f^2(w,\tau)\, dw + \frac{2\beta^*}{\alpha +2} \int_{-1}^1H^{\alpha -1}(w)[1-(1+2\alpha)w^2] f^{\alpha +2}(w,\tau)\, dw \\
&\le -2 \int_{-1}^1 H(w)\left| \frac{\partial}{\partial w} f(w,\tau)\right|^2\, dw 
+ \lambda^* \int_{-1}^1 f^2(w,\tau)\, dw + \frac{2\beta^*}{\alpha +2} \int_{-1}^1 f^{\alpha +2}(w,\tau)\, dw.
\end{equations}
Let us suppose first that $\lambda^*>0$. Then for any $t >0$, if $\|f(\cdot,t)\|_{L^2}^2  \le 2^{\alpha/(\alpha +2)}$, the $L^2$-norm of the solution at time $t $ is bounded. On the contrary, if $\|f(\cdot,t)\|_{L^2}^2  > 2^{\alpha/(\alpha +2)}$, from the H\"older inequality we get
\begin{equations}\label{LL}
 \int_{-1}^1 f^2(w,t)\, dw &\le
  \left(\int_{-1}^1 f^{2+\alpha}(w,t)\, dw\right)^{2/(\alpha +2)}\left(\int_{-1}^1 1\, dw\right)^{\alpha/(\alpha +2)}\\ &=  2^{\alpha/(\alpha +2)} \left(\int_{-1}^1 f^{2+\alpha}(w,t)\, dw\right)^{2/(\alpha +2)} \le  2^{\alpha/(\alpha +2)} \int_{-1}^1 f^{2+\alpha}(w,t)\, dw.
\end{equations}
Using this bound into \fer{L2} we obtain
\[
\frac d{dt}\int_{-1}^1 f^2(w,t)\, dw \le -2 \int_{-1}^1 H(w)\left| \frac{\partial  f(w,t)}{\partial w}\right|^2\, dw 
+ C_{\alpha,\beta^*}^{\lambda^*}  \int_{-1}^1 f^{\alpha +2}(w,t)\, dw.
\]
where
\[
C_{\alpha,\beta^*}^{\lambda^*} = 2^{\frac{\alpha}{\alpha +2}}\lambda^*+ \frac{2\beta^*}{\alpha +2}>0.
\]
If $\lambda^*<0$, from \eqref{L2} we get
\[
\dfrac{d}{d\tau}\int_{-1}^1 f^2(w,\tau) dw \le -2 \int_{-1}^1 H(w)\left| \dfrac{\partial f(w,\tau)}{\partial w}\right|^2dw + C_{\alpha,\beta}^0 \int_{-1}^1 f^{\alpha+2}(w,\tau)dw. 
\]
In general, if we define $C_{\alpha,\beta^*} = \max\left\{C_{\alpha,\beta^*}^{\lambda^*},C_{\alpha,\beta^*}^0\right\}$ we get
\begin{equations}\label{L3}
&\frac d{d\tau}\int_{-1}^1 f^2(w,\tau)\, dw \\
&\le -2 \int_{-1}^1 H(w)\left| \frac{\partial  f(w,\tau)}{\partial w}\right|^2\, dw 
+ C_{\alpha,\beta^*}  \int_{-1}^1 f^{\alpha +2}(w,\tau)\, dw.
\end{equations}
Since $0<\alpha < 2$ the new Gagliardo--Nirenberg with weight proven in Lemma \ref{lem:2} shows that
\[
\int_{-1}^1 H(w)\left| \frac{\partial  f(w,t)}{\partial w}\right|^2\, dw \ge C_{GN}^{3/(\alpha +1)} \left(\int_{-1}^1 f^{2+\alpha}(w,t)\, dw\right)^{3/(\alpha +1)}  \left(\int_{-1}^1 f_0(w)\, dw\right)^{-(\alpha+4)/(\alpha +1)}.
\]
On the other hand, inequality \fer{LL} implies
\begin{equations}\label{L4}
& C_{\alpha,\beta^*}  -2 C_{GN}^{3/(\alpha +1)} \left(\int_{-1}^1 f^{2+\alpha}(w,t)\, dw\right)^{\frac{2-\alpha}{\alpha +1}}  \left(\int_{-1}^1 f_0(w)\, dw\right)^{-\frac{\alpha+4}{\alpha +1}}  \\
&\le  C_{\alpha,\beta^*}  -  C_{GN}^{3/(\alpha +1)}2^{2/(\alpha +2)}\left(\int_{-1}^1 f^2(w,t)\, dw\right)^{\frac{2-\alpha}{\alpha +1}}  \left(\int_{-1}^1 f_0(w)\, dw\right)^{-\frac{\alpha+4}{\alpha +1}} .
\end{equations}

Thus, the right-hand side is negative when
\[
\int_{-1}^1 f^2(w,t)\, dw  \ge \left(\frac{C_{\alpha,\beta^*}}{ 2^{2/(\alpha +2)}{ C_{GN}^{3/(\alpha +1)}}}\right)^{\frac{\alpha +1}{2-\alpha}}.
\]
Therefore, the norm $\|f(\cdot,t)\|_{L^2}^2 $ is bounded for every $t >0$, since it can be bounded from above as follows
\[
\|f(\cdot,t)\|_{L^2}^2 \le \max\left\{\|f_0\|_{L^2}^2;  2^{\alpha/(\alpha +2)}; \left(\frac{C_{\alpha,\beta^*}}{ 2^{2/(\alpha +2)}{ C_{GN}^{3/(\alpha +1)}}}\right)^{\frac{\alpha +1}{2-\alpha}} \right\}.
\] 
The case $\alpha =2$ represents a limit case. Indeed, as $\alpha =2$, inequality \fer{L4} loses its dependence on $\|f(\cdot,t)\|_{L^2}$, and the sign of the last expression only depends on the value of the initial mass. Therefore, the $L^2$-norm of the solution remains bounded in time as soon as the initial mass is smaller than a constant, i.e.
\[
 \int_{-1}^1 f_0(w)\, dw \le  \left(\frac{C_{\alpha,\beta^*} }{ C_{GN}^{3/(\alpha +1)}2^{2/(\alpha +2)}}\right)^{\frac{\alpha+4}{\alpha +1}}.
\]

\section{Loss of  $L^2$-regularity for $\alpha >2$}\label{sect:2}

The  analysis of Section \ref{sec:regularity} shows that problems with $L^2$-regularity of the solution to \fer{eq:FP} can appear starting from the value $\alpha =2$, and the shape of the steady state \fer{steady} suggests that the solution will start to develop problems around the value $w=0$.  A first answer to this question can be obtained by resorting to an argument first introduced in \cite{T2}.

We remark that, given a nonnegative function $\phi$ in $L^2((-1,1))$  the second moment of $\phi$ controls its $L^1$-norm. This can be shown by considering that, for any given positive constant $R\le 1$ we have the bound
\begin{equations}\nonumber
\int_{-1}^1 \phi(w)\, dw  &=\int_{|w|\le R} \phi(w)\, dw + \int_{|w|>R} \phi(w)\, dw \\
&\le (2R)^{1/2} \left(\int_{-1}^1 \phi^2(w)\, dw\right)^{1/2} + \frac 1{R^2} \int_{-1}^1w^2 \phi(w)\, dw.
\end{equations}
Optimizing over $R$ we get
\be\label{b5}
\int_{-1}^1 \phi(w)\, dw \le 5\left(\frac 1{2 \sqrt 2}\right)^{4/5}\left(\int_{-1}^1 \phi^2(w)\, dw\right)^{2/5} \left(\int_{-1}^1 w^2 \phi(w)\, dw\right)^{1/5}.
\ee
Therefore, since the solution to equation \fer{eq:FP} preserves the mass, the boundedness of the $L^2$-norm of the solution prevents the second moment to be lower than a value determined by the size of the $L^2$-norm. Indeed, equation \eqref{b5} can be rewritten as
\be\label{bbound}
\int_{-1}^1 w^2 f(w,t)\, dw \ge \left(\frac 15\right)^5\frac{(2 \sqrt 2)^4\mu^5}{\left(\int_{-1}^1 f^2(w,t)\, dw\right)^{2}},
\ee
where $\mu = \int_{-1}^1 f_0(w)\, dw$ is the  initial mass. Looking at inequality \eqref{bbound} we get that we can exploit the relationship between the $L^2$-norm of the solution and its second moment to check  $L^2$-regularity of the solution to equation \fer{eq:FP}.  To this end, we will study in detail the evolution of the second moment. 

Wiring \eqref{eq:FP} in weak form and by considering as a test function $\varphi(w) = w^2$ we get the evolution of the second order moment. We remark that
\[
E(t) = \int_{-1}^1 w^2 f(w,t)dw \le \int_{-1}^1  f(w,t)dw = \mu.
\]
By direct computations we have
\begin{equation}
\label{eq:energy_evo}
\dfrac{d}{dt} E(t) = 2 \lambda \mu  - 2(\lambda+1)E(t) - 2\beta \int_{-1}^1 w^2 H^{\alpha}(w)f^{\alpha+1}(w,t)dw, 
\end{equation}
Hence, at each time, the evolution of the second order moment depends on higher order powers of the kinetic density $f(w,t)$. To have information on the evolution of the energy we need to quantify, in terms of the energy itself, the  integral in \fer{eq:energy_evo}. We prove the following  
\begin{lemma}
\label{prop:1}
Let $\alpha>2$. Then
\[
\int_{-1}^1 w^2 H^{\alpha}(w)f^{\alpha+1}(w,t)dw \ge \dfrac{(\mu-E(t))^{\frac{3\alpha}{2}}}{\left(c_\alpha d_\alpha + d_\alpha^{-2} \right)^{\frac{3\alpha}{2}}E^{\frac{\alpha-2}{2}}}, 
\]
with  $c_{\alpha} = \left( \frac{2\alpha}{\alpha-2} \right)^{\frac{\alpha}{\alpha+1}}>0$ and $d_\alpha  = \left[\frac{2(\alpha+1)}{c_\alpha(\alpha-2)} \right]^{\frac{1}{3}}>0$. 
\end{lemma}
\begin{proof}
For any $R\ge0$ we have
\begin{equation}
\label{eq:es1}
\begin{split}
&\int_{-1}^1 H^{\frac{\alpha}{\alpha+1}}(w) f(w,t)dw = \int_{|w|<R}H^{\frac{\alpha}{\alpha+1}}(w) f(w,t)dw + \int_{|w|\ge R}H^{\frac{\alpha}{\alpha+1}}(w) f(w,t)dw \le \\
&\quad \int_{|w|<R}H^{\frac{\alpha}{\alpha+1}}(w) f(w,t)dw +\dfrac{1}{R^2} \int_{|w|\ge R}w^2H^{\frac{\alpha}{\alpha+1}}(w) f(w,t)dw \le  \\
&\quad\int_{|w|<R}H^{\frac{\alpha}{\alpha+1}}(w) f(w,t)dw +\dfrac{1}{R^2} \int_{-1}^1w^2H^{\frac{\alpha}{\alpha+1}}(w) f(w,t)dw,
\end{split}
\end{equation}
In  inequality \fer{eq:es1} we defined $H(w) = (1-w^2)_+$ to overcome any a priori bounds on $R>0$. The first integral in \eqref{eq:es1} can be rewritten as follows
\[
\int_{|w|<R}H^{\frac{\alpha}{\alpha+1}}(w) f(w,t)dw = \int_{|w|<R} H^{\frac{\alpha}{\alpha+1}}(w) w^{-\frac{2}{\alpha+1}}\left( w^{\frac{2}{\alpha+1}}f(w,t) \right)dw.
\]
By applying Hölder's inequality with $p = \alpha+1$ and $p_* = \dfrac{\alpha+1}{\alpha}$ we get
\[
\int_{|w|< R} H^{\frac{\alpha}{\alpha+1}}(w)f(w,t)dw \le \left( \int_{|w|<R} w^{-\frac{2}{\alpha}}dw \right)^{\frac{\alpha}{\alpha+1}} \left( \int_{|w|<R} w^2 H^\alpha(w)f^{\alpha+1}(w,t)dw \right)^{\frac{1}{\alpha+1}}. 
\]
Since $\alpha>2$
\[
\int_{-R}^R \left(\dfrac{1}{w^2}\right)^{1/\alpha} dw = 2 \int_{0}^R \left(\dfrac{1}{w^2}\right)^{1/\alpha} dw = \dfrac{2\alpha}{\alpha-2} R^{\frac{\alpha-2}{\alpha}}. 
\]
Hence, we have 
\[
\int_{|w|< R} H^{\frac{\alpha}{\alpha+1}}(w)f(w,t)dw \le \left( \dfrac{2\alpha}{\alpha-2} \right)^{\frac{\alpha}{\alpha+1}}R^{\frac{\alpha-2}{\alpha+1}} \left( \int_{-1}^1 w^2 H^\alpha(w)f^{\alpha+1}(w,t)dw\right)^{\frac{1}{\alpha+1}}
\]. 
Therefore \eqref{eq:es1} can be written as follows
\begin{equation}
\label{eq:inR}
\begin{split}
&\int_{-1}^1 H^{\frac{\alpha}{\alpha+1}}(w) f(w,t)dw \le c_\alpha R^{\frac{\alpha-2}{\alpha+1}} \left(\int_{-1}^1 w^2 H^\alpha(w)f^{\alpha+1}(w,t)dw \right)^{\frac{1}{\alpha+1}}  \\
&\quad+ \dfrac{1}{R^2} \int_{-1}^1 w^2 H^{\frac{\alpha}{\alpha+1}}(w)f(w,t)dw, 
\end{split}
\end{equation}
being $c_{\alpha} = \left( \dfrac{2\alpha}{\alpha-2} \right)^{\frac{\alpha}{\alpha+1}}$. 
Next, we can optimise with respect to $R>0$ the function 
\[
g(R) = c_\alpha R^{\frac{\alpha-2}{\alpha+1}} \left(\int_{-1}^1 w^2 H^\alpha(w)f^{\alpha+1}(w,t)dw \right)^{\frac{1}{\alpha+1}}  + \dfrac{1}{R^2} \int_{-1}^1 w^2 H^{\frac{\alpha}{\alpha+1}}(w)f(w,t)dw. 
\]
The function $g(R)$ attains its unique minimum in
\[
R^* = \left[ \frac{2(\alpha+1)}{c_\alpha(\alpha-2)} \right]^{\frac{\alpha+1}{3\alpha}} \dfrac{\left(\displaystyle\int_{-1}^1 w^2H^{\frac{\alpha}{\alpha+1}}(w)f(w,t)dw\right)^{\frac{\alpha+1}{3\alpha}}}{\left(\displaystyle\int_{-1}^1 w^2 H^\alpha(w)f^{\alpha+1}(w,t)dw\right)^{\frac{1}{3\alpha}}}. 
\]
Therefore, from \eqref{eq:inR} we get
\begin{equation}
\label{eq:inRop}
\begin{split}
&\int_{-1}^1 H^{\frac{\alpha}{\alpha+1}}(w)f(w,t)dw \le \\
&\qquad (c_\alpha d_\alpha + d_\alpha^{-2}) \left( \int_{-1}^1 w^2 H^{\alpha}(w)f^{\alpha+1}(w,t)dw \right)^{\frac{2}{3\alpha}} \left( \int_{-1}^1 w^2 H^{\frac{\alpha}{\alpha+1}}(w) f(w,t)dw\right)^{\frac{\alpha-2}{3\alpha}}, 
\end{split}
\end{equation}
being $d_\alpha = \left[\dfrac{2(\alpha+1)}{c_\alpha(\alpha-2)} \right]^{\frac{1}{3}}$. Furthermore, since for any $\alpha>0$ we have
\[
0<\mu - E(t) = \int_{-1}^1 H(w)f(w,t)dw \le \int_{-1}^1 H^{\frac{\alpha}{\alpha+1}}(w)f(w,t)dw, 
\]
we may write \eqref{eq:inRop} as follows
\[
\int_{-1}^1 w^2 H^\alpha(w) f^{\alpha+1}(w,t)dw \ge \dfrac{(\mu-E(t))^{\frac{3\alpha}{2}}}{\left(c_\alpha d_\alpha + d_\alpha^{-2} \right)^{\frac{3\alpha}{2}}E^{\frac{\alpha-2}{2}}},
\]
which concludes the proof.
\end{proof}

The result of Lemma \ref{prop:1} allows to show that the evolution of the second order moment of the solution of  \eqref{eq:FP}, as given by in \eqref{eq:energy_evo}, attains  the value zero in finite time if the initial energy is sufficiently low. Indeed, incorporating the estimate of Proposition \ref{prop:1} into \eqref{eq:energy_evo} we obtain
\[
\begin{split}
\dfrac{d}{dt} E(t) \le 2\lambda \mu - 2(\lambda+1)E -2\beta\dfrac{(\mu-E(t))^{\frac{3\alpha}{2}}}{\left(c_\alpha d_\alpha + d_\alpha^{-2} \right)^{\frac{3\alpha}{2}}E^{\frac{\alpha-2}{2}}} 
\end{split}\] 
A stronger inequality holds, since $\lambda>0$ and $E\ge0$ for all $t\ge0$. Discarding the term $- 2(\lambda+1)E$, we reduce to 
\begin{equation}
\label{eq:ode_energy}
\dfrac{d}{dt}E\le 2\lambda \mu  -\dfrac{2\beta}{\left(c_\alpha d_\alpha + d_\alpha^{-2} \right)^{\frac{3\alpha}{2}}}\dfrac{(\mu-E(t))^{\frac{3\alpha}{2}}}{E^{\frac{\alpha-2}{2}}(t)}. 
\end{equation}
We can observe that, at time $t=0$  the following quantity 
\begin{equation}
\label{eq:phi0}
\Phi(E(0)) = 2\lambda \mu  -\dfrac{2\beta}{\left(c_\alpha d_\alpha + d_\alpha^{-2} \right)^{\frac{3\alpha}{2}}}\dfrac{(\mu-E(0))^{\frac{3\alpha}{2}}}{E^{\frac{\alpha-2}{2}}(0)}<0, 
\end{equation}
is negative provided $E(0)$ is sufficiently small. In Figure \ref{fig:1} we depict the relation between $E(0)$ and $\Phi(E(0))$ and for several values of $\alpha >2$, $\beta=1$ and for a fixed initial mass $\mu = 1$. We can argue that $0<E_c<\mu$ exists such that, for any $f(w,0)$ solution to \eqref{eq:FP} with $\int_{-1}^1 w^2 f(w,0)dw < E_c$, we have $\Phi(E(0))<0$. It is interesting to observe that, coherently with \eqref{eq:phi0}, for large values of $\lambda>0$ it is required a smaller second order moment $E(0)>0$ to guarantee that $\Phi(E(0))<0$. 

 \begin{figure}
 \centering
 \includegraphics[scale = 0.35]{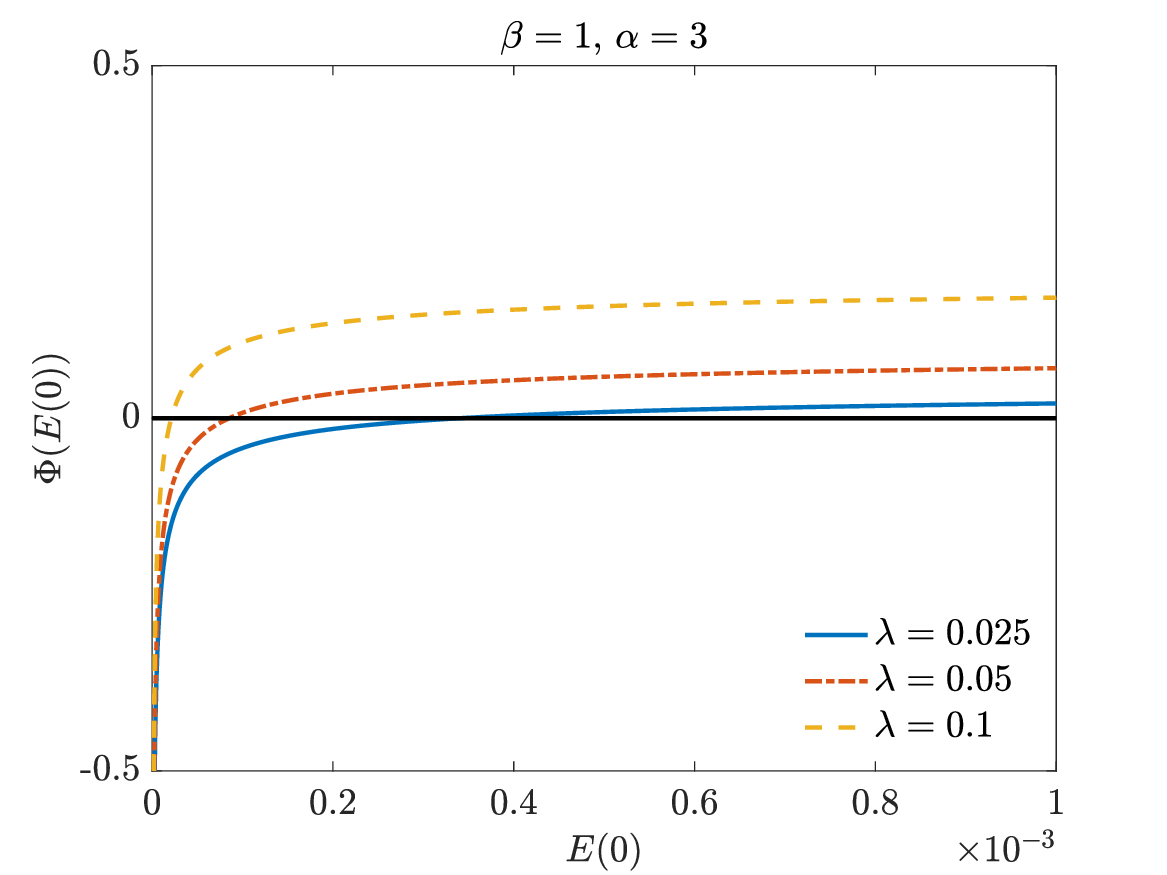}
 \includegraphics[scale = 0.35]{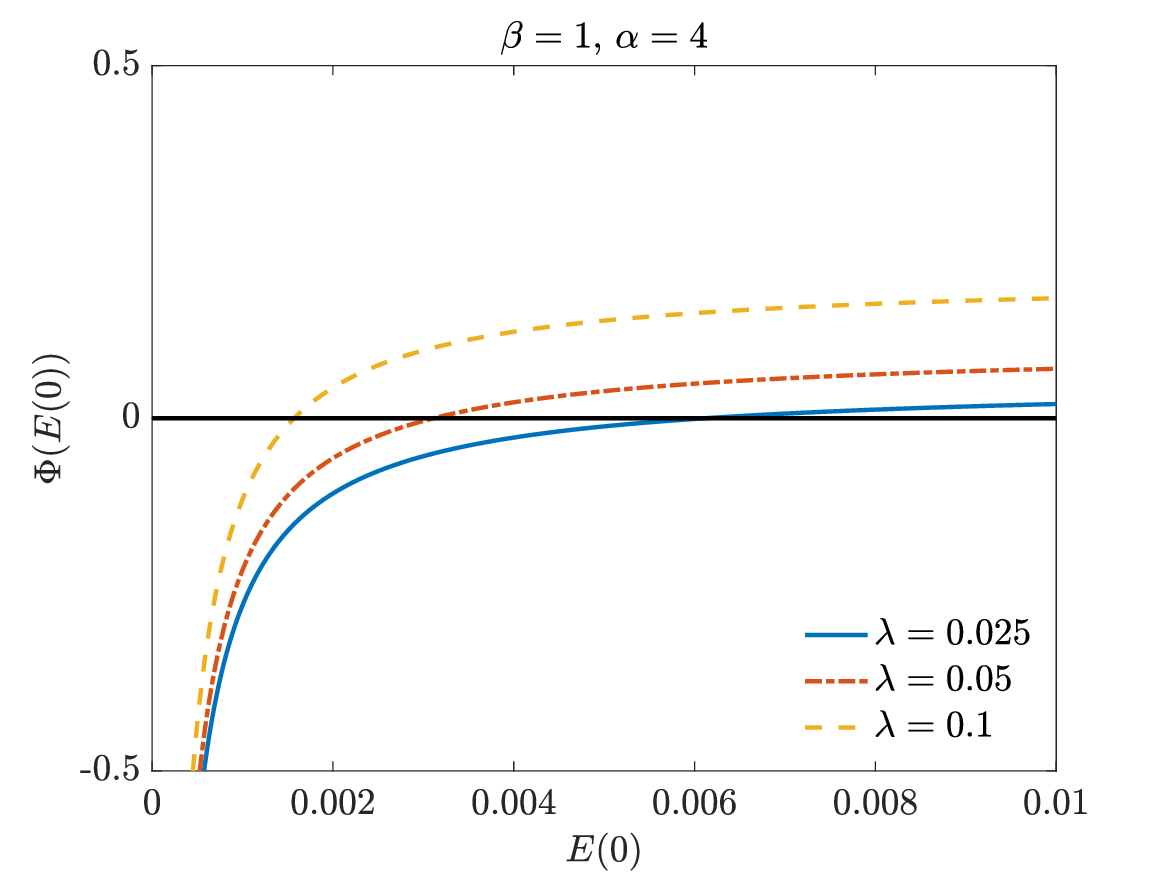}
 \caption{Values of $\Phi(E(0))$ in \eqref{eq:phi0} for a set of initial energies $E(0)>0$ and several values of the diffusion coefficient $\lambda =0.025,0.05,0.1$. We considered the case $\alpha = 3$ (left panel) $\alpha = 4 $ (right panel) and we fixed $\beta = 1$. }
 \label{fig:1}
 \end{figure}
 
Similarly, we can determine the evolution of the temperature 
\[
\mathcal T(t) = \dfrac{1}{\mu} \int_{-1}^1 w^2 f(w,t)dw, 
\]
from \eqref{eq:ode_energy} we get
\[
\dfrac{d}{dt} \mathcal T(t) \le 2\lambda - \dfrac{2\beta}{\left(c_\alpha d_\alpha + d_\alpha^{-2} \right)^{\frac{3\alpha}{2}}}\mu^{\alpha+1}\dfrac{(1-E(t))^{\frac{3\alpha}{2}}}{E^{\frac{\alpha-2}{2}}(t)}. 
\]
Since $\alpha+1 >0$ for any $\alpha>2$ and given second order moment $E(0)$, we can define the following quantity 
\[
\Psi(\mu) = 2\lambda - \dfrac{2\beta}{\left(c_\alpha d_\alpha + d_\alpha^{-2} \right)^{\frac{3\alpha}{2}}}\mu^{\alpha+1}\dfrac{(1-E(0))^{\frac{3\alpha}{2}}}{E^{\frac{\alpha-2}{2}}(0)}, 
\]
which is negative provided the initial mass is sufficiently large
\[
\mu > \left( \dfrac{\lambda}{\beta \left(c_\alpha d_\alpha + d_\alpha^{-2} \right)^{\frac{3}{2\alpha}} } \dfrac{E^{\frac{\alpha-2}{2}}(0)}{ (1-E(0))^{\frac{3\alpha}{2}}} \right)^{\frac{1}{\alpha+1}}. 
\]
It is worth to remark that the obtained relation between the initial mass and the diffusion coefficient $\lambda>0$ is such that the large the diffusion the large $\mu>0$ should be to guarantee $\Psi(\mu)<0$.

 Therefore, provided that $\Phi(E(0))<0$ or $\Psi(\mu)<0$,  corresponding to a sufficiently small initial energy or sufficiently large initial mass, respectively, from  \eqref{eq:ode_energy} we have 
 \[
 \dfrac{d}{dt} E(t) \le 2\lambda \mu - \dfrac{2\beta}{\left(c_\alpha d_\alpha + d_\alpha^{-2} \right)^{\frac{3}{2\alpha}}}  \dfrac{(\mu-E(0))^{\frac{3}{2\alpha}}}{E^{\frac{\alpha-2}{2\alpha^2}}}. 
 \]
 which is equivalent to 
 \begin{equation}
 \label{eq:diffE}
 \dfrac{d}{dt} E(t) \le 2\lambda \mu - \dfrac{\Lambda}{E^\gamma(t)} = -  \dfrac{b - 2\lambda \mu E^\gamma(t)}{E^\gamma(t)} \le -  \dfrac{\Lambda - 2\lambda \mu E^\gamma(0)}{E^\gamma(t)}
 \end{equation}
 with $\Lambda = \dfrac{2\beta (\mu-E(0))^{\frac{3}{2\alpha}}}{\left(c_\alpha d_\alpha + d_\alpha^{-2} \right)^{\frac{3}{2\alpha}}} $ and $\gamma = \frac{\alpha-2}{2\alpha^2}\in (0,\frac{1}{16})$ for any $\alpha>2$. Therefore, from \eqref{eq:diffE} we get
 \[
 E^{\gamma+1}(t) \le E^{\gamma+1}(0) - (\gamma+1)(\Lambda-2\lambda \mu E^{\gamma}(0))t, 
 \]
 and at time 
 \[ 
 \bar t = \dfrac{E^{\gamma+1}(0)}{(\gamma+1)(\Lambda-2\lambda \mu E^\gamma(0))}
 \]
we have $E(\bar t) = 0$ which is clearly in contradiction with inequality \eqref{bbound}, unless there is blow-up of $L^2$-norm of the solution. We observe that $\bar t>0$ provided the initial energy is sufficiently small and such that
\[
E(0) <\left( \dfrac{\Lambda}{2\lambda \mu}\right)^{1/\gamma}
\]
Let us better explain the meaning of the previous result. If $E(t) \to 0$ as $t \to \bar t$, also the mean value of the solution converges to zero as $t \to \bar t$. Indeed, owing to mass conservation, by Cauchy-Schwarz inequality
\[
\left| \int_{-1}^1 wf(w,t)\, dw \right| \le \int_{-1}^1 |w|f(w,t)\, dw \le E(t)^{1/2} \left( \int_{-1}^1 f_0(w)\, dw \right)^{1/2}.
\]
Consequently, as $t \to \bar t$ 
\[
\int_{-1}^1 \left( w -  \int_{-1}^1 wf(w,t)\, dw\right)^2f(w,t) \, dw \to 0.
\]
and the solution at $t=\bar t$ coincides with a Dirac delta concentrated in $w=0$. 

\section{Regularity of the solution when $\alpha >2$}
\label{sect:4}

Let us consider $\beta>0$ and $\alpha>2$. In this case, as showed in Section \ref{sect:2}, we cannot prove propagation of $L^2((-1,1))$ regularity since the solution may blow-up in finite time for a sufficiently large initial mass. Therefore, to get information about the possible behavior of the solution for large times, we study the evolution of the following weighted norm
\begin{equation}
\label{eq:norm}
\| f\|_{L^q_p} = \left( \int_{-1}^1 |w|^p f^q(w,t)dw \right)^{1/q}, \qquad p,q>0 
\end{equation}

\subsection{The case $q=2$}
In the case $q=2$, the norm defined in \eqref{eq:norm} corresponds to the $L^2_p$ norm. Writing \eqref{eq:FP} in weak form we get
\[
\begin{split}
&\dfrac{d}{dt}\int_{-1}^1 |w|^p f^2(w,t)dw =  \int_{-1}^1 2|w|^pf \dfrac{\partial}{\partial w}\left[ wf(1+\beta H^\alpha f^\alpha) + \lambda \dfrac{\partial}{\partial w}(Hf)\right] dw.
\end{split}\]
Integrating by parts we have
\begin{equation}
\label{eq:first_dis}
\begin{split}
\dfrac{d}{dt} \| f \|^2_{L^2_p} =& -2 \int_{-1}^1\dfrac{\partial}{\partial w} \left(|w|^pf \right) wf\left(1+\beta H^\alpha f^\alpha \right)dw \\
&-2\lambda \int_{-1}^1 \left(p|w|^{p-2}wf + |w|^p\dfrac{\partial}{\partial w} f \right) \dfrac{\partial}{\partial w} (Hf)dw \\
=&(-\lambda p^2 -(1-\lambda)p + 1-2\lambda)\|f \|_{L^2_p}^2 \\
&+ \dfrac{2\beta}{\alpha+2} \int_{-1}^1 \dfrac{\partial}{\partial w} \left( \dfrac{wH^{\alpha}}{|w|^{p(\alpha+1)}}\right)f^{\alpha+2}|w|^{p(\alpha+2)} dw \\
&  - 2\lambda \int_{-1}^1 H|w|^p \left(\dfrac{\partial f}{\partial w} \right)^2dw + \lambda p(p-1) \int_{-1}^1 |w|^{p-2}f^2dw. 
\end{split}
\end{equation}
The last integral is negative if $p \in (0,1)$. 
Also
\[
\dfrac{\partial}{\partial w} \left( \dfrac{wH^{\alpha}}{|w|^{p(\alpha+1)}}\right) <0,
\]
provided $p = \dfrac{\alpha-2}{\alpha+1}$ and $\alpha\ge3$.\\
Consequently, let us assume that $\alpha \ge 3$.
Then, since for $w \in (-1,1)$ 
\[
\int_{-1}^1 (1-|w|^{2-p})|w|^p \left(\dfrac{\partial f}{\partial w} \right)^2\, dw \le \int_{-1}^1 (1-w^2)|w|^p, \left(\dfrac{\partial f}{\partial w} \right)^2\, dw, 
\]
the result of Lemma \ref{lem:1bis},  guarantees that, proceeding as in Section \ref{sec:regularity},   the balance between the terms
\[
(-\lambda p^2 -(1-\lambda)p + 1-2\lambda)\|f \|_{L^2_p}^2   - 2\lambda \int_{-1}^1 H|w|^p \left(\dfrac{\partial f}{\partial w} \right)^2\, dw
\]
controls the upper bound of the weighted $L^2$-norm of $f$.\\
The following result holds
\begin{theorem}
Let $f = f(w,t)$ be a solution to \eqref{eq:FP} and let $\beta>0$. If the initial density $f(w,0) \in L^2_p$  then the $L^2_p$ norm of $f(w,t)$ defined in \eqref{eq:norm}, is uniformly bounded for all $t>0$ 
provided $p  = \dfrac{\alpha-2}{\alpha+1}$ and $\alpha\ge3$. 
\end{theorem}

\begin{remark}
We may notice therefore that the formation of a Dirac mass in finite time is not allowed in the introduced space $L_p^2((-1,1))$ unless $p = 1$, which, in turn, would correspond to the case $\alpha\to +\infty$. 
\end{remark}

\begin{remark}
An improvement of the obtained regularity result would require a sharper evaluation of the contribution of the diffusion. 
\end{remark}
\subsection{The case $1\le q <2$}
In the following we will write $q = 1+\gamma$, $\gamma \in (0,1)$ which corresponds to the norm $L^{1+\gamma}_p$ in \eqref{eq:norm}. Proceeding as before, from \eqref{eq:FP} we get
\[
\begin{split}
\dfrac{d}{dt} \int_{-1}^1 |w|^p f^{1+\gamma}(w,t)dw =& -(1+\gamma) \int_{-1}^1 \dfrac{\partial}{\partial w}(|w|^p f^\gamma) \left[wf(1+\beta H^\alpha f^\alpha) + \lambda\dfrac{\partial}{\partial w}(Hf)\right]dw \\
 =& (-\lambda p^2 - (1-\lambda)p + \gamma(1-2\lambda)) \int_{-1}^1 |w|^pf^{\gamma+1}dw  \\
&+ \dfrac{\gamma(1+\gamma)\beta}{\alpha+1+\gamma} \int_{-1}^1 |w|^pf^{\gamma} \dfrac{\partial}{\partial w} \left( \dfrac{wH^\alpha}{|w|^{p(\alpha+1)/\gamma}} \right)dw \\
& - \lambda \gamma(1+\gamma) \int_{-1}^1 |w|^p H f^{\gamma-1}\left(\dfrac{\partial f}{\partial w} \right)^2dw \\
& +\lambda p(p-1) \int_{-1}^1 |w|^{p-2}Hf^{\gamma+1}dw. 
\end{split}\]
As in the previous section, the coefficient of the last integral is negative if $p \in (0,1)$. \\
Also
\begin{equation}
\label{eq:termHq}
\dfrac{\partial}{\partial w} \left( \dfrac{wH^\alpha}{|w|^{p(\alpha+1)/\gamma}} \right) = \dfrac{H^{\alpha-1}}{|w|^{p(\alpha+1)/\gamma}}\left( H - 2\alpha w^2 - \dfrac{p(\alpha+1)}{\gamma}H \right). 
\end{equation}
Hence if $p = (\alpha-2)(\alpha+1)$ we have that \eqref{eq:termHq} is negative for any $\alpha > 2+\gamma$.
Last, since
\[
\int_{-1}^1 |w|^p H f^{\gamma-1}\left(\dfrac{\partial f}{\partial w} \right)^2\, dw = \frac 4{(\gamma +1)^2}\int_{-1}^1 |w|^p H \left(\dfrac{\partial f^{(\gamma +1)/2}}{\partial w} \right)^2\, dw,
\]
using once more the result of Lemma \ref{lem:1bis} we conclude that the balance between the terms
\[
(-\lambda p^2 - (1-\lambda)p + \gamma(1-2\lambda)) \int_{-1}^1 |w|^pf^{\gamma+1}dw - \lambda \gamma(1+\gamma) \int_{-1}^1 |w|^p H f^{\gamma-1}\left(\dfrac{\partial f}{\partial w} \right)^2\, dw
\]
 controls the upper bound of the weighted $L^2$-norm of $f^{(\gamma+1)/2}$. We have
 
 \begin{theorem}
Let $f = f(w,t)$ be a solution to \eqref{eq:FP} and let $\beta>0$. If the initial density $f(w,0) \in L_p^q$, where $1\le q <2$  then the $L_p^q$ norm of $f(w,t)$ defined in \eqref{eq:norm}, is uniformly bounded for all $t>0$ and,  provided $p  = \dfrac{\alpha-2}{\alpha+1}$,  when $\alpha\ge 1+ q$. 
\end{theorem}
 
\section*{Concluding remarks}
We studied the problem of regularity propagation in a kinetic model for consensus dynamics  characterized by a variable  diffusion coefficient which allows for condensation effects due to a superlinear drift . In particular, we showed the minimum degree of nonlinearity in the drift which is able to trigger the loss of classical $L^2$ regularity. To this end, we designed new Nash and Gagliardo-Nirenberg inequalities and we provided sufficient conditions to obtain blow-up of the solution in finite time. Thanks to suitable energy estimates, it is shown that the second order moment of the solution of the problem annihilates in finite time if the initial energy is sufficiently low. Equivalently, the temperature of the system attains zero for a sufficiently large initial mass. The obtained relationship between the sizes of the initial mass and diffusion is such that the large the diffusion the large should be the mass to produce loss of regularity. Finally, we showed that the solution of the problem is regular in a suitable weighted space if $\alpha>3$. To treat the case $\alpha\in (2,3]$ we considered a modified weighted space where the regularity is recovered. Future research will focus on this case together with the multidimensional case. 
\vfill\eject
\section{Appendix}\label{appendix}

In this Appendix we will collect the technical results which allow to extend the classical methods related to Fokker--Planck type equations with constant diffusion to the present case, characterized by a variable diffusion coefficient. Some of these results can be obtained by resorting to some recent studies which took motivations from the study of Fokker--Planck type equations for economic and social applications \cite{FPTT, FPTT1}. 

In the one-dimensional setting,  Nash inequality \cite{Nash} states that, for any function $\phi = \phi(w) \in L^1(\R)$ such that the distributional derivative $\phi'(w) \in L^2(\R)$, we get
\be\label{Nash}
\left( \int_\R |\phi(w)|^2 \, dw \right)^3 \le C \int_\R \left| \phi'(w)\right|^2 dw \left( \int_\R |\phi(w)| \, dw \right)^4,
\ee
for a given universal constant $C>0$. The optimal constant has been determined in \cite{CL} in any dimension of the physical space, by resorting to a tricky proof based on the sharp Poincar\'e inequality for radial functions, see also \cite{BDS}. As a matter of fact, Nash inequality is a key tool to control the boundedness in time of $L^p(\R)$ norms of the solution to Fokker--Planck equations in the presence of constant diffusion. Since the Fokker-Planck \eqref{eq:FP} has nonconstant diffusion coefficient $H(w)$, the classical Nash inequality cannot be directly applied. Likely, in a recent paper \cite{ToR} an improved one-dimensional version of Poincar\'e inequality has been shown to hold.
 For any given positive constant $R$ and for any differentiable function $\phi(w)$, $w \in D_R = (-R, R) $  the  one-dimensional Poincar\'e-type inequality with weight states that
 \be\label{p-Po}
\left \|  \phi(w) - \langle \phi\rangle \right\|_{L^2(D_R)}  \le \frac 1{\sqrt 2} \,\left \|   \sqrt{ R^2-w^2}\, \phi'(w)\right\|_{L^2(D_R)},
\ee
where
\[
\langle \phi \rangle= \frac 1{2R}\int_{D_R}\phi(v)\, dv
\]
is the average value of $\phi$  over the interval $D_R$, and the constant $1/2$ is sharp. We remark that for $R =1$, the weight function on the right-hand side of \fer{p-Po} coincides with the diffusion coefficient $H(w)$.\\
The validity of inequality \fer{p-Po} suggests a possible improvement of the Nash inequality in the bounded domain $D_1$ which allows to control the boundedness in time of $L^p(D_1)$ norms of the solution to the Fokker--Planck equation \fer{eq:FP} by methods similar to the ones valid for a constant coefficient of diffusion. The main result is given in the following:

\begin{lemma}\label{lem:1} Let  $\phi$ be an integrable function on the interval $D_1= (-1,1)$ such that the distributional derivative of $\phi$ is square integrable. Then the following Nash--type inequality with weight holds
\be\label{Nash1}
\left( \int_{-1}^1 |\phi(w)|^2 \, dw \right)^3 \le  C \int_{-1}^1 (1-w^2)\left| \phi'(w)\right|^2 dw \left( \int_{-1}^1 |\phi(w)| \, dw \right)^{4},
\ee
where
\be\label{co}
C= \frac {3^3}{2^{5}}.
\ee
\end{lemma}
\begin{proof}
We follow the idea of proof in \cite{CL}, by enlightening the main differences. For a Borel set $A$ with volume $|A|$, we define its spherically decreasing symmetric rearrangement $A^*$  by
\[
A^* =B(0,\rho),
\]
where $B(0,\rho)$ stands for the open ball with radius $\rho >0$ centered at the origin and $\rho$ is determined by the condition that $B(0,\rho)$ has volume $|A|$. 
Given  a measurable non-negative function $\phi=\phi(w)$, $w \in \R$, we define its spherically decreasing symmetric rearrangement $\phi^*$ by \cite{WM}
\be\label{ssr}
\phi^*(w) = \int_0^{+ \infty} [w\in B^*_t] \, dt
\ee
where $B_t=\{v:\phi(v) >t\}$, and the expression $[Condition]$ is equal to $1$ if \emph{Condition} is true and $0$ otherwise. Hence, the spherically symmetric rearrangment of $\phi$ is completely determined by the open balls $B_t^*$. 

Then, since for $p\ge 1$ $\| \phi\|_{L^p}= \| \phi^*\|_{L^p}$ \cite{WM}, provided
\be\label{goal}
\int_{-1}^1(1-w^2)\left| \frac{d \phi^*(w)}{dw}\right|^2 dw \le \int_{-1}^1(1-w^2)\left| \frac{d \phi(w)}{dw}\right|^2 dw,
\ee
 we may assume without loss of generality, while establishing inequality \fer{Nash1}, that $\phi =\phi^*$, and we can resort to  the same proof in \cite{CL}. To verify that inequality \fer{goal} holds, we remark that the integrals in \fer{goal} can be evaluated by substitution. Indeed, by setting for $x \in(-1,1)$
\[
f(x) = F(y); \quad dy = \frac 1{1-x^2}dx
\]
we obtain 
\be\label{id}
\int_{-1}^1 (1-x^2)|f'(x)|^2 \,dx = \int_\R |F'(y)|^2 \, dy 
\ee
where
\be\label{kkk}
F(y) = f(\gamma(y)); \quad \gamma(y)  = \tanh y = \frac{e^{y}-e^{-y}}{e^{y}+e^{-y}}.
\ee
Note that $\gamma(y)$ is a strictly increasing function that maps $\R$ onto $(-1,1)$.\\
Consequently, to verify inequality \fer{goal}, it is enough to prove that, for any given measurable non-negative function $f$ and a strictly increasing function $\theta$  mapping an interval $A\subseteq \R$ onto an interval $B \subseteq \R$ one has the identity
\be\label{kk}
f^*(\theta(y) ) = f^*\circ \theta(y) = (f\circ \theta)^*(y), \quad y \in A.
\ee
Indeed, when \fer{kk} holds true, on account of Faber--Krahn inequality \cite{Fab,Kra}
\be\label{F-K}
\int_\R \left| \frac{d F^*(y)}{dy}\right|^2 dy \le \int_\R\left| \frac{dF(y)}{dy}\right|^2 dy,
\ee
that implies, thanks to \fer{id}
\[
\int_{-1}^1(1-x^2)\left| \frac{d f^*(x)}{dx}\right|^2 dx \le \int_{-1}^1(1-x^2)\left| \frac{d f(x)}{dx}\right|^2 dx .
\]
To this end, as shown in Lemma 2.1 in \cite{WM},  we recall that  the following identity is valid
\be\label{lemm}
\{ x: f(x)>t\}^* = \{ x: f^*(x)>t\}.
\ee
From \eqref{lemm} we can show \fer{kk}. Indeed, since $y\to \theta(y)$ is a one-to-one mapping, it is invertible. 
Thus we get
\begin{equations}\nonumber
\{ y: f^*(\theta(y)) >t\} &= \{ y=\theta^{-1}(z): f^*(z) >t\} = \{ y=\theta^{-1}(z): f(z) >t\}^*\\
&= \{ y: f(\theta(y)) >t\}^* = \{ y: (f\circ \theta)^* (y))>t\}.
\end{equations}
Since \fer{kk} holds true, from now on, let us assume $\phi = \phi^*$. Next, let us fix any $r>0$, and let $R = r/(1+r)$ so that $0\le R < 1$. Let us define
\[
g(w) = \phi(w) \chi_{D_R}(w ), \quad h(w) = \phi(w) -g(w), 
\]
being $\chi_A(\cdot)$ the indicator function of the set $A \subseteq \mathbb R$ and $D_R = (-R,R)\subseteq(-1,1)$. 
Then we get
\[
\| \phi\|^2_{L^2} = \| g\|^2_{L^2}+ \| h\|^2_{L^2}.
\]
Since $\phi$ is radially decreasing, the mean value theorem shows that for any $w \in \mathbb R$
\[
h(w) \le \phi(R) \le \frac 1{2R} \|g\|_{L^1(D_R)}.
\]
Next, let $\langle g\rangle$ the average value of $g$ over the interval $D_R$, so that
\[
\langle g \rangle = \frac 1{2R} \|g\|_{L^1(D_R)}.
\]
By the Poincar\'e inequality \fer{p-Po} on the interval $D_R$ we have
\begin{equations}\label{prova}
 \| g\|^2_{L^2(D_R)} =& \int_{-R}^R \left(g(w) - \langle g\rangle \right)^2\,dw  + \int_{-R}^R \langle g\rangle^2\,dw \\
& \le \frac 12 \int_{-R}^R (R^2 -w^2) |\phi'(w)|^2 \, dw + \frac 1{2R} \| g\|^2_{L^1(D_R)} \\
 &=  \frac{R^2}2 \int_{-R}^R \left(1 -\frac{w^2}{R^2}\right) |\phi'(w)|^2 \, dw + \frac 1{2R} \| g\|^2_{L^1(D_R)} \\
  &\le \frac{R^2}2 \int_{-1}^1 \left(1 -{w^2}\right) |\phi'(w)|^2 \, dw + \frac 1{2R} \| g\|^2_{L^1(D_R)} .
\end{equations}
Then, since $\| \phi\|^2_{L^2(D_R)} = \| g\|_{L^2(D_R)}^2$ and 
\[
\| \phi\|_{L^1(D_R)}^2 \ge \|g \|_{L^1(D_R)}^2
\]
we conclude with the inequality
\[
\| \phi\|^2_{L^2} \le \frac 12\left[ {R^2} \int_{-1}^1 \left(1 -{w^2}\right)|\phi'(w)|^2 \, dw + \frac 1{R} \| \phi\|^2_{L^1}\right].
\]
Minimizing over $R$ we obtain inequality \fer{Nash1}.
\end{proof}
The result of Lemma \ref{lem:1} can be easily generalized to cover a class of weight functions different from $1-w^2$. We prove
\begin{lemma}\label{lem:1bis} Let $\phi$ be an integrable function on the interval $D_1= (-1,1)$ such that the distributional derivative of $\phi$ is square integrable. Then, for any given $0<p<1$, the following Nash--type inequality with weight holds
\be\label{Nash2bis}
\left(\int_{-1}^1 |\phi(w)|^4 \, dw\right)^3  \le  C_p\,\int_{-1}^1 (1-|w|^{2-p})|w|^p\left| \phi'(w)\right|^2 dw \left( \int_{-1}^1 |\phi(w)| \, dw \right)^4,
\ee
where the constant $C_p$ is defined as follows
\[
C_p =\dfrac{3^3}{2^4(2-p)(1-p)}.
\] 
\end{lemma}
\begin{proof}
We recall that, as recently proven in \cite{FPTT1}, Theorem 2.1, if $X$ is a random variable distributed with density $f_\infty(w)$, $w\in (i_-,i_+)$, satisfying the equality
\be\label{eqFPTT}
\frac d{dw}(P(w) f_\infty(x)) + Q(w) f_\infty(w) = 0,
\ee
where $P(w), Q(w)$ are $C^1$  functions such that $P(w)\ge 0$, $Q'(w) >0$ and 
\[
\lim_{w\to i_-}Q(w) <0, \quad \lim_{w\to i_+} Q(w)>0,
\]
then, for any smooth function $\phi$ such that $\phi(X)$ has finite variance, the following Poincar\'e-type inequality with weight holds
\be\label{var}
Var[\phi(X)] \le E\left( \frac{P(X)}{Q'(X)}[\phi'(X)]^2\right).
\ee
As in Lemma \ref{lem:1}, let us fix any $r>0$, and let $R = r/(1+r)$ so that $0\le R < 1$. Then, let $X$ be a random variable uniformly distributed on the interval $(-R,R)$, so that $f_\infty (w)= 1/(2R)$. If we choose, for a given $0<p<1$ and $w \in(-R,R)$ 
\[
P(w) = R^{2-p} - |w|^{2-p}, \quad Q(x) = (2-p) |w|^{-p} w,
\]
the pair $P,Q$ satisfies all the conditions listed above. Then, since
\[
Q'(w) = (2-p)(1-p) |w|^{-p} >0
\]
inequality \fer{var} holds, taking the form
\be\label{p-gen}
\left \|  \phi(w) - \langle \phi\rangle \right\|_{L^2(D_R)}^2  \le \frac 1{(2-p)(1-p)} \,\left \|   \sqrt{(R^{2-p} - |w|^{2-p})|w|^p }\, \phi'(w)\right\|_{L^2(D_R)}^2,
\ee
where $D_R =(-R,R)$, and
\[
\langle \phi \rangle= \frac 1{2R}\int_{D_R}\phi(v)\, dv
\]
is the average value of $\phi$  over the interval $D_R$. Note that, letting $p\to 0$, we recover inequality \fer{p-Po}. Also, the coefficient in \fer{p-gen} blows up as $p\to 1$.\\
From now on, we can repeat the proof of lemma \ref{lem:1} to conclude. The only difference is concerned with the application of Faber-Krahn inequality \fer{F-K}, that in this case would by given by 
\[
\int_{-1}^1(1-|x|^{2-p})|x|^p\left| \frac{d f^*(x)}{dx}\right|^2 dx \le \int_{-1}^1(1-|x|^{2-p})|x|^p\left| \frac{d f(x)}{dx}\right|^2 dx .
\]
A simple way to re-apply the proof of this point, as given in Lemma \ref{lem:1}, is to evaluate the integrals splitting the domain  $(-1,1)$ in the two intervals $(-1,0)$ and $(0,1)$. Clearly, if $x \in (0,1)$, since
\[
\frac 1{(1-x^{2-p})x^p} = \frac 1x + \frac{x^{-p} x}{1-x^{2-p}} = \frac d{dx} \log\frac x{(1-x^{2-p})^{1/(2-p)}},
\]
 the substitution 
\[
f(x) = F(y); \quad dy = \frac 1{(1-x^{2-p})x^p}dx
\]
leads to
\[
x(y) = \dfrac{ e^{y-y_o}}{\left( 1+ e^{(2-p)(y-y_0)}\right)^{1/(2-p)}},
\]
namely to a function that is strictly increasing for $y\in \R$ so that the identity \fer{kk} holds. Analogous result in the interval $(-1,0)$. \\ This concludes the proof.
\end{proof}

Another important consequence of Lemma \ref{lem:1} is that it can be used to obtain a certain class of Gagliardo--Nirenberg inequalities with weight, which are useful to control the regularity of the solution to equation \fer{eq:FP}. The following holds

\begin{lemma}\label{lem:2} Let  $\phi$ be an integrable function on the interval $D_1= (-1,1)$ such that the distributional derivative of $\phi$ is square integrable. Then the following Nash--type inequality with weight holds
\be\label{Nash2}
\int_{-1}^1 |\phi(w)|^2 \, dw  \le  D \int_{-1}^1 (1-w^2)\left| \phi'(w)\right|^2 dw \left( \int_{-1}^1 |\phi(w)| \, dw \right)^{2},
\ee
where
\[
D=\dfrac{\sqrt{2}}{2} \left( \dfrac{5}{3}\right)^{5/2}\left(\frac 32\right)^{3}.
\]
\end{lemma}

\begin{proof} 
For any given positive constant $\rho$ let $R= \rho(1+\rho)$, so that $0\le R < 1$, and let $w \in D_R$.  Let $r(w, t)=wt$ so that $r(w,0) =0$ and $r(w,1) = w$. Let $\dot r(w,t)$ denote the partial derivative of $r$ with respect to $t$. Thanks to the gradient theorem, for any given (smooth) function $\phi(w)$ we have
\[
\phi^2(w) - \phi^2(0) = \int_0^1 [\phi^2(r(w,t))]' \dot r(w,t) dt = 2 \int_0^1 \phi'(wt) \phi(wt) w \,dt.
\]
so that, by Cauchy--Schwarz inequality
 \be\label{ook}
 \left| \phi^2(w) - \phi^2(0)\right|  \le  2 \left( \int_0^1 |\phi'(wt)|^2 |w|^{3/2}\, dt\right)^{1/2} \left( \int_0^1 |\phi(wt)|^2 |w|^{1/2}\, dt\right)^{1/2}. 
 \ee
 Let us choose $\phi \in L^4(D_1)$. Then, since for $w \not= 0$
 \[
 \int_0^1 |\phi(wt)|^2 \, dt = \frac 1w \int_0^w \phi^2(v) \, dv,
 \]
 the Cauchy-Schwarz inequality implies
 \begin{equations}\label{quarto}
|w|^{1/2} \left| \int_0^1 |\phi(wt)|^2  \, dt \right| &=  |w|^{-1/2} \left| \int_0^w \phi^2(v) \, dv\right| \\
&|w|^{-1/2} \left( \int_0^w \phi^4(v) \, dv\right)^{1/2} \left( \int_0^w 1 \, dv\right)^{1/2} \le  \left( \int_{-R}^R \phi^4(v) \, dv\right)^{1/2}.
 \end{equations}
On the other hand, we may observe that, once we define 
\[
\langle \phi^2 \rangle= \frac 1{2R}\int_{D_R}\phi^2(v)\, dv,
\]
the following property of the variance does hold
 \[
\int_{-R}^R  \phi^4(w)\frac 1{2R} \ dw -   \langle \phi^2\rangle^2  \leq \int_{-R}^R (\phi^2(w)-\phi^2(0))^2 \frac 1{2R}\ dw. 
\] 
Hence, inequalities  \fer{ook} and \fer{quarto} imply
\be\label{cher}
 \int_{-R}^R  \phi^4(w)\frac 1{2R} \ dw -   \langle \phi^2\rangle^2  \leq 2\left( \int_{-R}^R \phi^4(v) \, dv\right)^{1/2}  \int_{-R}^R  \frac 1{2R}\int_0^1 |\phi'(wt)|^2 |w|^{3/2}\, dt\, dw.
 \ee
At this point we use the argument in Theorem 18 of \cite{FPTT}.   Let $w \in (-R,R)$, and let us set $\kappa(w) = R^{3/2} - |w|^{3/2}$. Thanks to the identity 
\[
\frac d{dw}(R^{3/2} - |w|^{3/2}) = - \frac 32 \frac w{|w|^{1/2}} 
\]
which implies 
\[
|w|^{3/2} = -\frac 23 w \frac {d\kappa(w)}{dw},
\]
we can write
    \begin{equations}\label{par}
  & \int_{-R}^R |w|^{3/2} \frac 1{2R}  \int_0^1 | \phi'(wt)|^2 \ {dt} \ dw = 
  - \frac 23  \int_{-R}^R  \frac 1{2R} w \frac{d\kappa(w) }{d w}  \int_0^1 [ \phi'(wt)]^2 \ {dt} \, dw   \\
  &\leq\frac 23 \int_{-R}^R \kappa(w)  \frac1{2R} \frac{\partial}{\partial w} \left[ w \int_0^1 | \phi'(wt)|^2 \, {dt}\right] \, dw.
 \end{equations}
 In \fer{par}, we used  integration by parts to get the last line. Indeed here the border term satisfies
 \be\label{ss}
-  \left. w\kappa(w)  \int_0^1 [ \phi'(wt]^2 \ {dt} \right|_{-R}^{+R} \le 0.
 \ee
 Next, since for any given function $\psi(wt)$ one has the identity
 \[
 \frac \partial{\partial t} \left( t \, \psi(wt )\right) = \frac{\partial}{\partial w} \left[ w \,\psi(wt)\right],
 \]
it holds 
 \[
\frac{\partial}{\partial w} \left[ w \int_0^1 | \phi'(wt|^2 \, {dt}\right] = \int_0^1  \frac d{dt} \left( t |\phi'(wt)|^2 \right) \, dt = 
|\phi'(w)|^2.
 \]
 Substituting into \fer{par} gives 
 \be\label{3k}
  \int_{-R}^R  \frac 1{2R}\int_0^1 [\phi'(wt)]^2 |w|^{3/2}\, dt\, dw \le \frac 1{3R} \int_{-R}^R (R^{3/2} - |w|^{3/2})[\phi'(w)]^2\, dw.
\ee
Now, considering that for $|w| \le R$
\[
R^{3/2} - |w|^{3/2} = \frac{R^{3/2} - |w|^{3/2} }{R^{2} - |w|^{2} }(R^{2} - |w|^{2} ) \le \frac 1{\sqrt R} (R^{2} - |w|^{2} ) ,
\]
we finally obtain from \fer{cher} the inequality
\be\label{cher2}
 \int_{-R}^R  \phi^4(w)\frac 1{2R} \ dw -   \langle \phi^2\rangle^2  \leq \frac 2{3 R^{3/2}} \left( \int_{-R}^R \phi^4(v) \, dv\right)^{1/2}
 \int_{-R}^R (R^{2} - |w|^{2})|\phi'(w)|^2\, dw.
\ee
Let us now apply the proof of Lemma \ref{lem:1} to the function $\phi^2$. The Poincar\'e inequality leading to \fer{prova} is here substituted by the Poincar\'e type inequality \fer{cher2} to give
\[
\| \phi\|^4_{L^4} \le { \frac 43 R^{3/2}}\| \phi\|^2_{L^4} \int_{-1}^1 \left(1 -{w^2}\right) |\phi'(w)|^2 \, dw + \frac 1{2R} \| \phi^2 \|^2_{L^1}.
\]
Minimizing over $R$ we obtain
\be\label{nn}
\left( \int_{-1}^1 |\phi(w)|^4 \, dw \right)^2 \le \dfrac{\sqrt{2}}{2} \left( \dfrac{5}{3}\right)^{5/2}\int_{-1}^1 (1-w^2)\left| \phi'(w)\right|^2 dw \left( \int_{-1}^1 |\phi(w)|^2 \, dw \right)^3.
\ee
Using now the upper bound given by Nash inequality \fer{Nash1} on the last integral in \fer{nn} we finally get inequality \fer{Nash2}.
\end{proof}

We may observe that the result of Lemma \ref{lem:2} can be considered to obtain a larger class of weighted new Gagliardo-Nirenberg inequalities. Indeed, for $2 < p <4$, from the classical interpolation inequality 
\be\label{inter}
\int_1^1 |\phi(w)|^p \, dw \le \left(\int_1^1 |\phi(w)|^4 \, dw \right)^{(p-1)/3}\left(\int_1^1 |\phi(w)|\, dw \right)^{(4- p)/3}
\ee
coupled with the Nash-type inequality with weight \fer{Nash2} states the Gagliardo-Nirenberg inequality with weight
\be\label{GN}
\int_{-1}^1 |\phi(w)|^p \, dw  \le  C_{GN}\,\left( \int_{-1}^1 (1-w^2)\left| \phi'(w)\right|^2 dw\right)^{(p-1)/3}  \left( \int_{-1}^1 |\phi(w)| \, dw \right)^{(p+2)/3},
\ee
where the positive constant $C_{GN} $ is explicitly given in terms of the constants $C$ of Lemma \ref{lem:1} and the 
constant $D$ of Lemma \ref{lem:2}
\be\label{co2}
C_{GN} = (C\cdot D)^{(p -1)/3}.
\ee

\begin{remark} The previous proofs are closely dependent from the weighted Poincar\'e-type inequality \fer{p-Po}, and from the monotonicity properties of the function $H$. This suggests that similar inequalities could be derived also in other cases, where  weighted Poincar\'e inequalities, analogous to \fer{p-Po},  have been proven to hold with a weight function sharing the same properties. The possibility to get these new inequalities is presently under study.
\end{remark}
}
\section*{Acknowledgements}
Both the authors are member of  GNFM (Gruppo Nazionale di Fisica Matematica) of INdAM, Italy. M.Z. acknowledge support of PRIN2022PNRR project No.P2022Z7ZAJ and of the European Union - NextGeneration EU.


\begin{thebibliography}{99}

\bibitem{APTZ} G. Albi, L. Pareschi, G. Toscani, M. Zanella. Recent advances in opinion modeling: control and social influence. In N. Bellomo, P. Degond, E. Tadmor (eds) \emph{Active Particles Volume 1, Advances in Theory, Models and Applications}, Modeling and Simulation in Science and Technology, Birkh{\"a}user--Springer, 2017. 

\bibitem{BAGT}
N. Ben Abdallah, I. M. Gamba, and G. Toscani. On the minimization problem of sub-linear convex functionals. \emph{Kinet. Relat. Models}, 4(4):857-871, 2011.

\bibitem{BL}
S. Bianchini, G. M. Leccese. Existence and blow-up for non-autonomous scalar conservation laws with viscosity. \emph{J. Math. Anal. Appl.}, 542(1):128761, 2025. 

\bibitem{BDS}
E. Bouin, J. Dolbeault, C. Schmeiser, A variational proof of Nash’s inequality. \emph{Atti Accad. Naz. Lincei Cl. Sci. Fis. Mat. Natur.}, 31 (2020), no. 1, pp. 211–223


\bibitem{BHW}
M. Burger, J. Haskovec, M.-T. Wolfram, Individual based and mean-field modelling of direct aggregation. \emph{Phys. D}, 260:145--158, 2013.

\bibitem{CDTZ}
E. Calzola, G. Dimarco, G. Toscani, M. Zanella. Emergence of condensation patterns in kinetic equations for opinion dynamics. \emph{Phys. D}, in press. 

\bibitem{CL}
E. Carlen, M. Loss. Sharp constant in Nash's inequality. \emph{Int. Math. Res. Not.} , \textbf{7}: 213--215, 1993.

\bibitem{CFRT} 
J. A. Carrillo, M. Fornasier, J. Rosado, G. Toscani. Asymptotic flocking dynamics for the kinetic Cucker--Smale model. \emph{SIAM J. Math. Anal.}, 42(1): 218--236, 2010. 

\bibitem{CRS}
J. A. Carrillo, J. Rosado, F. Salvarani. 1D nonlinear {F}okker-{P}lanck equations for fermions and bosons. \emph{Appl. Math. Lett.}, 21(2):148--154, 2008.




\bibitem{CS}
F. Cucker and S. Smale, Emergent behavior in flocks, \emph{IEEE Trans. Automat. Control}, 52:852--862, 2007.

\bibitem{DG}
M. H. DeGroot, Reaching a consensus. \emph{J. Amer. Statist. Assoc.}, 69 (1974), pp. 118--121.

\bibitem{DFWZ}
B. Düring, J. Franceschi, M.-T. Wolfram, M. Zanella. Breaking consensus in kinetic opinion formation models on graphons. \emph{J. Nonlin. Sci.}, 34:79,2024.

\bibitem{D_etal}
B. D\"uring, P. Markowich, J. F. Pietschmann, M. T. Wolfram, Boltzmann and Fokker-Planck equations modelling opinion formation in the presence of strong leaders, \emph{Proc. R. Soc. Lond. Ser. A}, 465 (2112), pp. 3687--3708.

\bibitem{Fab}
G. Faber. Beweis, daft unter allen homogenen Membranen yon gleicher Fl\"ache und gleicher Spannung die kreisf\"ormige den tiefsten Grundton gibt. \emph{Sitzungsber. Math.--Phys. K1. Bayer. Akad. Wiss. M\"unch.}: 162--172, 1923.

\bibitem{FLST}
S. Fornaro, S. Lisini, G. Savar\'e, G. Toscani. Measure valued solutions of sub-linear diffusion equations with a drift term. \emph{Discrete and Continuous Dynamical Systems A.}, \textbf{32} (5):1675--1707,  2012.

\bibitem{FPTT}
G. Furioli, A. Pulvirenti, E. Terraneo, G. Toscani. Fokker--Planck equations in the modelling of socio-economic phenomena. \emph{Math. Models Methods Appl. Scie.} \textbf{27} (1): 115--158,  2017.

\bibitem{FPTT2}
 G. Furioli, A. Pulvirenti, E. Terraneo, G. Toscani. Wright-Fisher-type equations for opinion formation, large time behavior and weighted logarithmic-Sobolev inequalities. \emph{Ann. IHP, Analyse Non Linéaire} \textbf{36}: 2065--2082, 2019.

\bibitem{FPTT1}
G. Furioli, A. Pulvirenti, E. Terraneo, G. Toscani. Fokker-Planck equations and one-dimensional functional inequalities for heavy tailed densities. \emph{Milan J. Math.}, \textbf{90}: 177--208, 2022. 




\bibitem{guidolin}
P. L. Guidolin, L. Sch\"utz, J. S. Ziebell, J. P. Zingano. Global existence results for solutions of general conservative
advection-diffusion equations in $\R$. \emph{J. Math. Anal. Appl.}, 515: 126361, 2022. 

\bibitem{HK}
R. Hegselmann, U. Krause. Opinion dynamics and bounded confidence: models, analysis and simulation. \emph{J. Artif. Soc. Soc. Simul.}, 5,3, 2002.

{\bibitem{Hopf}
K. Hopf.
Singularities in$L^1$-supercritical Fokker--Planck equations: A qualitative analysis.
\emph{Ann. Inst. H. Poincar\'e, 
Anal. Non Lin\'eaire} \textbf{41}  357--403, 2024.
}

\bibitem{KQ}
G. Kaniadakis and P. Quarati. Kinetic equation for classical particles obeying an exclusion principle. \emph{Phys. Rev. E}, 48:4263–4270, 1993.

\bibitem{Kra}
E. Krahn. \"Uber eine yon Rayleigh formulierte Minimaleigenschaft des Kreises. \emph{Math. Ann.} \textbf{94}:  97--100, 1925.

\bibitem{LBL}
C. Le Bris, P.-L. Lions. Existence and uniqueness of solution to Fokker-Planck type equations with irregular coefficients. \emph{Commun. Part. Differ. Equat.}, 33(7):1272--1317, 2008.

\bibitem{MT}
S. Motsch, E. Tadmor. Heterophilious dynamics enhances consensus. \emph{SIAM Rev.}, 56(4): 577--621, 2013.  

\bibitem{Nash}
J. Nash. Continuity of solutions of parabolic and elliptic equations. \emph{Amer. J. Math} \textbf{80} : 931--954, 1958.

\bibitem{PTTZ}
L. Pareschi, G. Toscani, A. Tosin, M. Zanella. Hydrodynamic models of preference formation in multi-agent societies. \emph{J. Nonlin. Sci.}, 29(6):2761-2796, 2019. 

\bibitem{T}
G. Toscani. Kinetic models of opinion formation. \emph{Commun. Math. Sci.}, 4:481--496, 2006. 

\bibitem{T2}
G. Toscani. Finite time blow up in Kaniadakis-Quarati model of Bose-Einstein particles. \emph{Commun. Partial Differ. Equ.}, 37(1):77--87, 2012.

\bibitem{ToR}
G. Toscani. One-dimensional Barenblatt-type solutions and related inequalities. \emph{Ricerche di Matematica}, \textbf{73}(Suppl.1): 309-321 (2023)

\bibitem{WM} 
L. Wang and M. Madiman. Beyond the entropy power inequality, via rearrangements. \emph{IEEE Trans. Inform. Theory}, \textbf{60}, (9):5116--5137, 2014

\end{thebibliography}
\end{document}